\documentclass[]{amsart}   
\usepackage{amsmath}
\usepackage[mathscr]{eucal} 
\usepackage{amssymb}
\usepackage{latexsym}
\usepackage{amsthm} 
\theoremstyle{plain}
\newtheorem{theorem}{Theorem}[section]
\newtheorem{lemma}[theorem]{Lemma}

\newtheorem{corollary}[theorem]{Corollary}

\newcommand{\supp}{\mathop{\mathrm{supp}}\nolimits}

\theoremstyle{definition} 

\newtheorem{definition}[theorem]{Definition}
\numberwithin{equation}{section}

\theoremstyle{remark}

\def\Xint#1{\mathchoice
{\XXint\displaystyle\textstyle{#1}}%
{\XXint\textstyle\scriptstyle{#1}}%
{\XXint\scriptstyle\scriptscriptstyle{#1}}%
{\XXint\scriptscriptstyle\scriptscriptstyle{#1}}%
\!\int}
\def\XXint#1#2#3{{\setbox0=\hbox{$#1{#2#3}{\int}$}
\vcenter{\hbox{$#2#3$}}\kern-.5\wd0}}

\def\dashint{\Xint-}

\title[Characterization of parabolic Hardy spaces]
{Characterization of parabolic Hardy spaces by 
Littlewood-Paley functions}  
\author{Shuichi Sato} 
 
\begin{document} 

\address{Department of Mathematics,
Faculty of Education, Kanazawa University, Kanazawa 920-1192, Japan}
\email{shuichi@kenroku.kanazawa-u.ac.jp}
\begin{abstract} 
We consider Littlewood-Paley functions associated with non-isotropic dilations.   We prove that they can be used to 
characterize the parabolic Hardy spaces of Calder\'{o}n-Torchinsky. 
\end{abstract}
  \thanks{
  \endgraf
    }
\thanks{The author is partly supported
by Grant-in-Aid for Scientific Research (C) No. 25400130, Japan Society for the  Promotion of Science.}

\maketitle 

\section{Introduction}  
Let $P$ be an $n\times n$ real matrix such that 
$$\langle Px,x\rangle \geq \langle x,x\rangle \quad \text{for all $x\in 
\Bbb R^n$, }
$$  
where $\langle x,y\rangle=x_1y_1+\dots + x_ny_n$ is the inner product in 
$\Bbb R^n$. 
Let $\gamma=\text{{\rm trace} $P$}$. Define a dilation group 
$\{A_t\}_{t>0}$ on $\Bbb R^n$ by $A_t=t^P=\exp((\log t) P)$. 
It is known that $|A_tx|$ is strictly increasing as a function of $t$ on 
$\Bbb R_+=(0,\infty)$ 
for $x\neq 0$, where $|x|=\langle x,x\rangle^{1/2}$.  
Define a norm function $\rho(x)$ to be the unique positive real 
number $t$ such that $|A_{t^{-1}}x|=1$ when $x\neq 0$ and  $\rho(0)=0$.  
Then $\rho(A_tx)=t\rho(x)$, $t>0$, $x\in \Bbb R^n$, 
$\rho \in C^\infty(\Bbb R^n\setminus\{0\})$ 
 and the following results are known 
(see \cite{CT, Ca}): 
\begin{enumerate} 
\item[(P.1)]  $\rho(x+y)\leq \rho(x)+\rho(y)$; 
\item[(P.2)]  $\rho(x)\leq 1$ if and only if $|x|\leq 1$; 
\item[(P.3)]   $|x|\leq \rho(x)$ if $|x|\leq 1$; 
\item[(P.4)]   $|x|\geq \rho(x)$ if $|x|\geq 1$; 
\item[(P.5)]   $|A_tx|\geq t|x|$ if $t\geq 1$ for all $x\in \Bbb R^n$; 
\item[(P.6)] $|A_tx|\leq t|x|$ if $0<t\leq 1$ for all $x\in \Bbb R^n$.   
\end{enumerate} 
 Similarly, we can consider a norm function $\rho^*(x)$ associated with 
 the dilation group $ \{A_t^*\}_{t>0}$, where $A_t^*$ denotes the adjoint 
 of $A_t$.  We have properties analogous to those for $\rho(x)$, $A_t$ above.  
\par 
Let 
\begin{equation}\label{lpop}
g_{\varphi}(f)(x) = \left( \int_0^{\infty}|f*\varphi_t(x)|^2
\,\frac{dt}{t} \right)^{1/2}   
\end{equation} 
be the Littlewood-Paley function on $\Bbb R^n$, 
where $\varphi_t(x)=t^{-\gamma}\varphi(A_t^{-1}x)$ and   
 $\varphi$ is a function in  $L^1(\Bbb R^n)$  such that  
\begin{equation}\label{cancell}
\int_{\Bbb R^n} \varphi (x)\,dx = 0.  
\end{equation} 
For $L^p$ boundedness of $g_{\varphi}$, $1<p<\infty$, see \cite{BCP, Sa2}. 
\par  
We say that a tempered distribution $f$ belongs to the parabolic
$H^p$ if 
$\|f\|_{H^p}=\|f^*\|_p<\infty$ , where $f^*(x)=\sup_{t>0}|\Phi_t*f(x)|$ 
and $\|f^*\|_p=\|f^*\|_{L^p}$, 
with $\Phi\in \mathscr S(\Bbb R^n)$ satisfying $\int \Phi(x)\, dx=1$, 
$\supp(\Phi)\subset \{|x|\leq 1\}$    
(see \cite{CT, CT2},  \cite{FeS2}). We have denoted by $\mathscr S(\Bbb R^n)$ 
the Schwartz class of rapidly decreasing smooth functions on $\Bbb R^n$. 
\par  
Let $\varphi \in L^1(\Bbb R^n)$. We consider the non-degeneracy condition: 
\begin{equation}\label{nondege1}  
\sup_{t>0}|\hat{\varphi}(A_t^*\xi)|>0 \quad \text{for all $\xi\neq 0$, }
\end{equation} 
where the Fourier transform is defined as 
$$
\hat{f}(\xi)=\mathscr F(f)(\xi)=\int_{\Bbb R^n} f(x)e^{-2\pi i\langle x,\xi
\rangle}\, dx. 
$$  
In this note we shall prove the following. 
\begin{theorem}\label{T1.1}  
Let $\varphi$ be a function in $\mathscr S(\Bbb R^n)$ satisfying 
\eqref{cancell} and 
\eqref{nondege1}.  Let $0<p\leq 1$. Then if $f\in H^p$, we have 
\begin{equation*}  
c_1\|f\|_{H^p}\leq \|g_{\varphi}(f)\|_p\leq c_2\|f\|_{H^p} 
\end{equation*} 
with some positive constants $c_1, c_2$ independent of $f$.  
\end{theorem}  
\par 
We recall some related results when $P=E$ (the identity matrix) and 
$g_\varphi$  is defined by $\varphi_t(x)=t^{-n}\varphi(t^{-1}x)$ in 
\eqref{lpop}. 
Then Theorem \ref{T1.1} is known (see \cite{U} and also \cite{G} for some 
background materials). 
Let 
$Q(x)= [(\partial/\partial t) P(x,t)]_{t=1}$, 
where 
$$P(x,t)=c_n \frac{t}{(|x|^2+t^2)^{(n+1)/2}}, 
\quad c_n=\frac{\Gamma((n+1)/2)}{\pi^{(n+1)/2}}, $$ 
is the Poisson kernel associated with the upper half space 
$\Bbb R^n \times (0, \infty)$ (see \cite[Chap. I]{SW}). 
We note that $\hat{Q}(\xi)=-2\pi|\xi|e^{-2\pi|\xi|}$.  
Then it is also known that 
\begin{equation}\label{hpequiv}
c_1\|f\|_{H^p}\leq \|g_Q(f)\|_{p}\leq c_2\|f\|_{H^p} 
\end{equation}   
for $f\in H^p(\Bbb R^n)$, $0<p<\infty$, 
with positive constants $c_1, c_2$ (see \cite{FeS2} and also \cite{U}).   
In \cite[Chap. 7]{FoS} we can find a relation between Hardy spaces and 
Littlewood-Paley functions associated with the heat kernel on homogeneous 
groups. 
\par 
 Uchiyama \cite{U} gave a proof of the first inequality of \eqref{hpequiv} 
for $0<p\leq 1$ by methods of real analysis without the use of special 
properties of the Poisson kernel such as  harmonicity, a semigroup property.  
Applying a similar argument, \cite{U} proved the first inequality of 
the conclusion of Theorem \ref{T1.1} (when $P=E$) for $0<p\leq 1$: 
\begin{equation}\label{fineq} 
\|f\|_{H^p}\leq c\|g_\varphi(f)\|_{p}. 
\end{equation}  
\par 
For a function $F$ on $\Bbb R^n$ and positive real numbers $N, R$,
define the Peetre maximal function $F^{**}_{N,R}$ by 
\begin{equation*}\label{pmax}
F^{**}_{N,R}(x)=\sup_{y\in \Bbb R^n}\frac{|F(x-y)|}{(1+R|y|)^N}     
\end{equation*} 
(see \cite{P}).  
Then, it is known that the maximal function $F^{**}_{N,R}$ 
can be used along with well-known arguments to prove \eqref{fineq} 
for $0<p\leq 1$ when 
$\varphi \in \mathscr S(\Bbb R^n)$ satisfies a non-degeneracy condition 
and  the condition $\supp(\hat{\varphi}) 
\subset \{a_1\leq |\xi|\leq a_2\}$ for $a_1, a_2>0$.   
\par 
In \cite{Sa3}, \eqref{fineq} was proved for $f$ in a dense subspace of 
$H^p(\Bbb R^n)$ and  for $\varphi$ in a class of functions  
including $Q$ and a general $\varphi \in  \mathscr S(\Bbb R^n)$, 
without the restriction on $\supp(\hat{\varphi})$ above, 
with \eqref{cancell} and \eqref{nondege1}   
by applying  a vector valued inequality 
related to the Littlewood-Paley theory.  
The proof of the vector valued inequality is based on an application of 
the maximal function $F^{**}_{N,R}$.  
This proof of \eqref{fineq}   was discovered by \cite{Sa3}.  
\par 
The purpose of this note is to generalize the methods of \cite{Sa3} to the case
 of the parabolic Hardy spaces and get the characterization of the parabolic 
 Hardy spaces  in terms of Littlewood-Paley functions (Theorem \ref{T1.1}). 
 \par 
In Section 2, 
we shall prove an analogue of the vector valued inequality in \cite{Sa3} for 
the general dilation group $\{A_t\}$ (Theorem \ref{T2.3}). 
We shall consider $g_\varphi$ for $\varphi$ in a class of functions which 
includes those $\varphi$ considered in Theorem \ref{T1.1} and prove 
\eqref{fineq} in the case of the parabolic $H^p$ for such $\varphi$ and for $f$ in a dense subspace of $H^p$ as an application of Theorem \ref{T2.3} 
(Corollary \ref{C3.1}).  Theorem \ref{T2.3} will be 
stated more generally than needed for the proof of 
Corollary \ref{C3.1} as weighted vector valued inequalities. 
\par 
In Section 3, a proof of Corollary \ref{C3.1} will be given by applying 
Theorem \ref{T2.3} and an atomic decomposition for Hardy spaces. 
Also, Theorem \ref{T1.1} will be derived from Corollary \ref{C3.1}.

\section{Weighted vector valued inequalities with non-isotropic dilations}  
We need a partition of unity on $\Bbb R^n \setminus\{0\}$ associated with 
$\varphi$ satisfying \eqref{nondege1}. 

\begin{lemma}\label{L2.1}  
Suppose that $\varphi$ is a function in 
$L^1(\Bbb R^n)$ satisfying  \eqref{nondege1}. We assume that $\hat{\varphi} 
\in C^\infty(\Bbb R^n\setminus\{0\})$.   Then, there 
exist $b_0\in (0,1)$ and $r_1, r_2>0$, $r_1<r_2$, such that for any 
$b\in [b_0, 1)$ there exists a function $\eta$ with the 
following properties$:$ 
\begin{enumerate}
\item[(1)] $\eta\in C^\infty(\Bbb R^n); $  
\item[(2)] $\hat{\eta}\in C^\infty(\Bbb R^n)$ and 
$\supp \hat{\eta} \subset \{r_1<\rho^*(\xi)<r_2\};$   
\item[(3)]
$\sum_{j=-\infty}^\infty \hat{\varphi}(A_{b^j}^*\xi)
\hat{\eta}(A_{b^j}^*\xi) =1$ \quad for $\xi\in \Bbb R^n\setminus\{0\}$. 
\end{enumerate}  
\end{lemma}  
When $P$ is the identity matrix, this is in \cite[Chap. V]{ST}; see also 
\cite{CT}.
\begin{proof}[Proof of Lemma $\ref{L2.1}$.]  
Let $S^{n-1}=\{\xi: |\xi|=1\}$. 
Since $\hat{\varphi}$ is continuous, by a compactness 
argument we can find a finite family $\{I_h\}_{h=1}^L$ of compact intervals in 
$(0,\infty)$ such that 
\begin{equation*} 
\inf_{\xi\in S^{n-1}}\max_{1\leq h\leq L}\inf_{t\in I_h}
|\hat{\varphi}(A_t^*\xi)|^2\geq c   
\end{equation*}  
with a positive constant $c$.  
\par 
We observe that there exists $b_0\in (0,1)$ such that if $b\in [b_0,1)$, 
$t>0$ and $1\leq h\leq L$,  then we have $b^jt\in I_h$ for 
some $j\in \Bbb Z$ (the set of integers).  This can be seen by taking $b_0
=\max_{1\leq h\leq L}(a_h/b_h)$, where $I_h=[a_h,b_h]$. 
\par 
Consider an interval $[m,H]$ in $(0, \infty)$ such that 
$\cup_{h=1}^L I_h\subset [m,H]$ and 
choose $\theta\in C_0^\infty(\Bbb R)$ such that $\theta=1$ on $[m,H]$, 
$\supp \theta\subset [m/2,2H]$, $\theta\geq 0$.  Define 
$$\Psi(\xi)=\sum_{j=-\infty}^\infty\theta(b^j\rho^*(\xi))
|\hat{\varphi}(A_{b^j}^*\xi)|^2.  $$ 
Then $\Psi(\xi) \geq c>0$ for $\xi\neq 0$.  
Note that $\Psi(A_{b^k}^*\xi)=\Psi(\xi)$ for $k\in \Bbb Z$.  
Let  
$$\hat{\eta}(\xi)=\theta(\rho^*(\xi))
\overline{\hat{\varphi}(\xi)} 
\Psi(\xi)^{-1}\quad \text{for $\xi\neq 0$} $$ 
and $\hat{\eta}(0)=0$, where $\overline{\hat{\varphi}(\xi)}$ denotes the 
complex conjugate. 
Then, $\eta$ satisfies all the properties stated in the lemma.   
This completes the proof. 
\end{proof}  
\par 
To state our results, we introduce a class of functions. 
\begin{definition} 
Let $\varphi$ be as in Lemma $\ref{L2.1}$.  Then, we say 
$\varphi \in B$ if the following conditions are satisfied: 
\begin{gather} \label{alpha} 
\varphi\in C^1(\Bbb R^n), \quad \partial_{k}\varphi\in L^1(\Bbb R^n), 
\quad 1\leq k\leq n; 
\\ 
\label{beta} 
|\hat{\varphi}(\xi)|\leq C|\xi|^\epsilon \quad \text{for some $\epsilon>0;$} 
\\ 
\label{beta+} 
|\partial_\xi^\alpha \hat{\varphi}(\xi)|\leq 
C_{\alpha,\tau}|\xi|^{-\tau}  
\quad \text{outside a neighborhood of the origin, }
\intertext{for all $\alpha$ and $\tau>0$  with a constant $C_{\alpha,\tau}$, 
where $\alpha=(\alpha_1, \dots, \alpha_n)$ is a multi-index, 
$\alpha_j\in \Bbb Z$, $\alpha_j\geq 0$, $|\alpha|=\alpha_1+\dots+\alpha_n$  
and $\partial_\xi^\alpha=
\partial_{\xi_1}^{\alpha_1}\dots \partial_{\xi_n}^{\alpha_n}$, 
$\partial_{\xi_k}=\partial/\partial_{\xi_k}$.  }  
\notag 
\end{gather}  
\end{definition}  
If $\psi\in \mathscr S(\Bbb R^n)$ with \eqref{cancell} and \eqref{nondege1}, 
then clearly  $\psi\in B$ .  
\par 
Let $\varphi\in B$.  
Then \eqref{alpha} implies that 
\begin{equation}\label{nearorigin} 
\mathscr F(\partial_{k}\varphi)(\xi)=\Xi_k(\xi)\hat{\varphi}(\xi), 
\quad 1\leq k\leq n, 
\end{equation} 
where $\Xi_k(\xi)=2\pi i\xi_k$ and 
$\partial_k=\partial_{x_k}$. 
Let $b\in [b_0,1)$ and $\eta$ be as in Lemma \ref{L2.1}.  
Define 
$$\hat{\zeta}(\xi)=1-\sum_{j\geq 0} \hat{\varphi}(A_{b^j}^*\xi)
\hat{\eta}(A_{b^j}^*\xi).     $$  
Then $\supp(\hat{\zeta})\subset \{\rho^*(\xi)\leq r_2\}$, 
$\hat{\zeta}=1$ 
in $\{\rho^*(\xi)<r_1\}$ and  by \eqref{nearorigin} we have 
\begin{align*} 
\mathscr F(\partial_{k}\varphi)(\xi)&=\sum_{j\geq 0}
\mathscr F(\partial_{k}\varphi)(\xi)\hat{\varphi}(A_{b^j}^*\xi) 
\hat{\eta}(A_{b^j}^*\xi)
+ \hat{\zeta}(\xi)\mathscr F(\partial_{k}\varphi)(\xi)  
\\ 
&=\sum_{j\geq 0} \hat{\varphi}(A_{b^j}^*\xi)
\mathscr F(\alpha^{(b^j)}_{(k)})(A_{b^j}^*\xi) 
+  \hat{\varphi}(\xi) \mathscr F(\beta_{(k)})(\xi), 
\end{align*}
where 
$\alpha^{(b^j)}_{(k)}(x)=(\partial_{k}\varphi)_{b^{-j}}*\eta(x)$ 
and 
$\mathscr F(\beta_{(k)})(\xi)=\hat{\zeta}(\xi)\Xi_k(\xi)$.   
\par 
Thus we have 
\begin{equation}\label{ineq1} 
|F(\partial_{k}\varphi,f)(x,t)|\leq   
\sum_{j\geq 0}|F(\alpha^{(b^j)}_{(k)}*\varphi,f)(x,b^jt)|
+|F(\beta_{(k)}*\varphi,f)(x,t)|,   
\end{equation} 
where $f\in \mathscr S(\Bbb R^n)$ and $F(\psi,f)(x,t)=f*\psi_t(x)$.  
We also write $F(\psi,f)(x,t)= F_\psi(x,t)$ when $f$ is fixed.   
\par 
Let 
\begin{equation}\label{czero} 
C_0(\partial_{k}\varphi,t,L,x)= 
(1+\rho(x))^L\left|\int \mathscr F(\partial_{k}\varphi)(A_{t^{-1}}^*\xi)
\hat{\eta}(\xi)
e^{2\pi i\langle x, \xi\rangle}\, d\xi\right|,  
\end{equation} 
with $t>0, L\geq 0$. 
Then 
\begin{equation*} 
|\alpha^{(b^j)}_{(k)s}(x)|
= C_0(\partial_{k}\varphi,b^j,L,A_s^{-1}x)s^{-\gamma}(1+\rho(x)/s)^{-L}     
\end{equation*}  
for $j\in \Bbb Z$. 
Similarly,    
\begin{equation*} 
|\beta_{(k)s}(x)|=D(\Xi_k,L,A_s^{-1}x)s^{-\gamma}(1+\rho(x)/s)^{-L},     
\end{equation*}  
with 
\begin{equation}\label{d}
D(\Xi_k,L,x)= (1+\rho(x))^L\left|\int \hat{\zeta}(\xi)
\Xi_k(\xi)e^{2\pi i\langle x, \xi\rangle}\, d\xi\right|. 
\end{equation}  
Put $C(\partial_{k}\varphi,j, L,x)=C_0(\partial_{k}\varphi,b^j,L,x)$, 
$j\in \Bbb Z$.    
Define 
\begin{gather}\label{c} 
C(\partial_{k}\varphi,j, L)= 
\int_{\Bbb R^n} C(\partial_{k}\varphi,j, L,x)\, dx, \quad j\in \Bbb Z,    
\\ 
\label{d2} 
D(\Xi_k, L)= \int_{\Bbb R^n} D(\Xi_k, L,x)\, dx.  
\end{gather} 
We can consider  $C(\psi,j, L)$ for other $\psi \in L^1(\Bbb R^n)$ by 
\eqref{czero} and \eqref{c} with $\psi$ in place of $\partial_{k}\varphi$. 
We also write  $C(\psi,j, L)=
C_\varphi(\psi,j, L)$, 
$D(\Xi_k, L)=D_\varphi(\Xi_k, L)$ to indicate that these quantities are 
based on $\varphi$. 
We have $D(\Xi_k, L), C(\partial_{k}\varphi,j, L)<\infty$ for all $j, L$, 
which can be seen from Lemma \ref{L2.8} below. 
\par 
Let 
\begin{equation}\label{pmax}
F^{**}_{N,R}(x)=\sup_{y\in \Bbb R^n}\frac{|F(x-y)|}{(1+R\rho(y))^N}     
\end{equation} 
for a function $F$ on $\Bbb R^n$ and positive real numbers $N, R$. 
We need the following result in proving Theorem \ref{T2.3} below.  
\begin{lemma}\label{L2.2} 
Let $\varphi\in B$,  $b\in [b_0,1)$, $N>0$.   
Then we have 
\begin{multline*}  
\label{ineq3} 
F(\partial_{k}\varphi,f)(\cdot,t)^{**}_{N, t^{-1}}(x) 
\leq C\sum_{j\geq 0}
C(\partial_{k}\varphi, j,N)b^{-jN} 
F(\varphi,f)(\cdot,b^jt)^{**}_{N, (b^jt)^{-1}}(x) 
\\ 
+CD(\Xi_k,N) F(\varphi,f)(\cdot,t)^{**}_{N, t^{-1}}(x)     
\end{multline*} 
for $1\leq k\leq n$, where $f\in \mathscr S(\Bbb R^n)$. 
\end{lemma} 
\begin{proof} 
By \eqref{ineq1} we have 
\begin{multline*}  
|F_{\partial_{k}\varphi}(z,t)| 
\\ 
\leq C \sum_{j\geq 0}\int |F_\varphi(y,b^jt)|
\left(1+\frac{\rho(z-y)}{b^jt}\right)^{-N}
C(\partial_{k}\varphi, j, N,A_{b^jt}^{-1}(z-y)) (b^jt)^{-\gamma}\, dy 
\\ 
+C\int |F_\varphi(y,t)|
\left(1+\frac{\rho(z-y)}{t}\right)^{-N}D(\Xi_k, N,A_{t^{-1}}(z-y)) 
t^{-\gamma}\, dy.    
\end{multline*} 
Multiplying both sides of the inequality by $(1+\rho(x-z)/t)^{-N}$ and noting 
that 
\begin{equation*} 
\left(1+\frac{\rho(z-y)}{b^jt}\right)^{-N}
\left(1+\frac{\rho(x-z)}{t}\right)^{-N}\leq C_{A,N}b^{-Nj}
\left(1+\frac{\rho(x-y)}{b^jt}\right)^{-N}
\end{equation*}  
for any $x, y, z\in \Bbb R^n$ and $t>0$ when $b^j\leq A$,  we have 
\begin{align*}  
&|F_{\partial_{k}\varphi}(z,t)|(1+\rho(x-z)/t)^{-N}
\\ 
&\leq C \sum_{j\geq 0}b^{-Nj}\int |F_\varphi(y,b^jt)|
\left(1+\frac{\rho(x-y)}{b^jt}\right)^{-N}
C(\partial_{k}\varphi, j, N,A_{b^jt}^{-1}(z-y))(b^jt)^{-\gamma}\, dy 
\\ 
&\phantom{ \leq\, }+ C\int |F_\varphi(y,t)|
\left(1+\frac{\rho(x-y)}{t}\right)^{-N}D(\Xi_k,N,A_t^{-1}(z-y))
t^{-\gamma}\, dy 
\\ 
&\leq C \sum_{j\geq 0}b^{-Nj}F_\varphi(\cdot,b^jt)^{**}
_{N, (b^jt)^{-1}}(x)\int C(\partial_{k}\varphi, j, N,A_{b^jt}^{-1}(z-y))
(b^jt)^{-\gamma}\, dy 
\\ 
&\phantom{ \leq\, }+C F_\varphi(\cdot,t)^{**}
_{N, t^{-1}}(x)\int D(\Xi_k, N,A_t^{-1}(z-y))t^{-\gamma}\, dy 
\\ 
&\leq C \sum_{j\geq 0}C(\partial_{k}\varphi,j,N)b^{-Nj}
F_\varphi(\cdot,b^jt)^{**}_{N, (b^jt)^{-1}}(x) 
+CD(\Xi_k,N) F_\varphi(\cdot,t)^{**}_{N, t^{-1}}(x). 
\end{align*}
Taking the supremum in $z$ over $\Bbb R^n$, 
we reach the conclusion. 
\end{proof} 
\par 
Let $\varphi\in B$. Then we have 
\begin{equation}\label{ineq5} 
\sup_{j\geq 0}C_\varphi(\nabla\varphi,j,L)b^{-\tau j}<\infty,  
\end{equation} 
for all $L, \tau>0$, 
where we write  
$\nabla\varphi=(\partial_{1}\varphi, \dots , \partial_{n}\varphi)$, 
 $C_\varphi(\nabla\varphi,j,L)=\sum_{k=1}^n 
C_\varphi(\partial_{k}\varphi, j,L)$ and    
\begin{equation}\label{ineq5+} 
D_\varphi(L)<\infty  
\end{equation}  
for all $L>0$, where $D_\varphi(L)=\sum_{k=1}^n D_\varphi(\Xi_k,L)$.   
Let $\psi\in \mathscr S (\Bbb R^n)$.  
Then, we also note that 
\begin{equation}\label{ineq4}
\sup_{j: b^j\leq r_2}C_\varphi(\psi,j,L)b^{-\tau j}<\infty \quad 
\text{for any $L, \tau>0$.} 
\end{equation} 
These results are in Lemma \ref{L2.8} below, which will be used 
in what follows. 
\par   
We consider a ball in $\Bbb R^n$ with center $x$ and radius $t$  
relative to $\rho$ defined by 
$$B(x,t)=\{y\in\Bbb R^n: \rho(x-y)<t\}. $$   
We say that a weight function $w$ belongs to the class $A_p$, 
 $1<p< \infty$, of Muckenhoupt if 
 $$[w]_{A_p}=  
 \sup_B \left(|B|^{-1} \int_B w(x)\,dx\right)\left(|B|^{-1} \int_B
w(x)^{-1/(p-1)}dx\right)^{p-1} < \infty, $$
where the supremum is taken over all balls $B$ in $\Bbb R^n$ and $|B|$ denotes 
the Lebesgue measure of $B$.  Let $A_\infty=\cup_{p>1} A_p$.  
Also, we define the class $A_1$ to be the family of weight functions $w$ 
such that $M(w)\leq Cw$ almost everywhere.  
 We denote by $[w]_{A_1}$ the infimum of all such $C$.  
Here, $M$ is the Hardy-Littlewood maximal operator relative to $\rho$ 
$$M(f)(x)=\sup_{x\in B}|B|^{-1}\int_B|f(y)|\,dy,   $$ 
where the supremum is taken over all balls $B$ in $\Bbb R^n$ containing $x$. 
(See \cite{C, GR}.)   
\par 
For a weight $w$,  we denote by $\|f\|_{p,w}$ the weighted $L^p$ norm 
$$\left(\int_{\Bbb R^n} |f(x)|^p w(x)\, dx\right)^{1/p}. $$
Then we have the following result. 
\begin{theorem}\label{T2.3} 
Let $\varphi \in B$. 
Suppose that $0<p, q<\infty$ and $w \in A_{\infty}$. 
Let $\psi\in \mathscr S(\Bbb R^n)$. 
Suppose that $\hat{\psi}=0$ in a 
neighborhood of  the origin. 
 Then 
$$\left\|\left(\int_0^\infty|f*\psi_t|^q\, \frac{dt}{t}
\right)^{1/q}\right\|_{p,w} 
\leq C\left\|\left(\int_0^\infty|f*\varphi_t|^q\, 
\frac{dt}{t}\right)^{1/q}\right\|_{p,w}   $$  
for $f\in \mathscr S(\Bbb R^n)$ with a positive constant $C$ independent of 
$f$.  
\end{theorem}  
\par 
To prove Theorem \ref{T2.3} we first show the following.  
\begin{lemma}\label{L2.4}  
Let $0<q<\infty$, $N>0$. 
Suppose that $\varphi\in B$, $f\in \mathscr S(\Bbb R^n)$. Then 
\begin{equation*}\label{ineq10}
\int_0^\infty F(\varphi,f)(\cdot,t)^{**}_{N, t^{-1}}(x)^q \, \frac{dt}{t} \leq 
C\int_0^\infty M(|f*\varphi_t|^r)(x)^{q/r}\, \frac{dt}{t}, \quad r=\gamma/N.   
\end{equation*}  
\end{lemma} 
 We apply the next result to show Lemma \ref{L2.4}.   
\begin{lemma}\label{Lpar1}  
Let $N=\gamma/r$, $r>0$ and let $\varphi\in L^1(\Bbb R^n)$, 
$f\in \mathscr S(\Bbb R^n)$. 
Suppose that $\varphi$ satisfies \eqref{alpha}.  
Then  
$$(f*\varphi_t)^{**}_{N,t^{-1}}(x)\leq C\delta^{-N} 
M(|f*\varphi_t|^r)(x)^{1/r} +
C\delta |f*(\nabla \varphi)_t|^{**}_{N,t^{-1}}(x)  $$  
for all $\delta \in (0,1]$ with a constant $C$ independent of $\delta$ 
and $t>0$, where $f*(\nabla\varphi)_t=(f*(\partial_{1}\varphi)_t, \dots , 
f*(\partial_{n}\varphi)_t)$. 
\end{lemma}
To prove Lemma \ref{Lpar1}, we use the following. 
\begin{lemma}\label{Lpar2}
Let $F\in C^1(\Bbb R^n)$ and $R>0$, $N=\gamma/r$, $r>0$. Then  
$$F^{**}_{N,1}(x)\leq C\delta^{-N} M(|F|^r)(x)^{1/r} +C\delta 
|\nabla F|^{**}_{N,1}(x)  $$  
for all $\delta \in (0,1]$ with a constant $C$ independent of $\delta$, 
where $F^{**}_{N,R}$ is as in \eqref{pmax}. 
\end{lemma} 

\begin{proof}
Let 
$\dashint_{B(x,t)} f(y)\,dy=|B(x,t)|^{-1}\int_{B(x,t)} f(y)\,dy$. 
Then, for $\delta\in (0,1], r>0$ and $x, z\in \Bbb R^n$ we write 
\begin{equation*} 
|F(x-z)|
=\left(\dashint_{B(x-z,\delta)}|F(y)+(F(x-z)-F(y))|^r\, dy\right)^{1/r}.  
\end{equation*}
This is bounded by 
\begin{equation*}
 C_r\left(\dashint_{B(x-z,\delta)}|F(y)|^r\, dy\right)^{1/r} 
+C_r\left(\dashint_{B(x-z,\delta)}|F(x-z)-F(y)|^r\, dy\right)^{1/r},  
\end{equation*} 
where $C_r=1$ if $r\geq 1$ and $C_r=2^{-1+1/r}$ if $0<r< 1$.  Thus we have 
\begin{equation*}
 |F(x-z)| \leq C_r\left(\dashint_{B(x-z,\delta)}|F(y)|^r\, dy\right)^{1/r} 
+C_r  \sup_{y:\rho(x-z-y)<\delta}|x-z-y| |\nabla F(y)|.  
\end{equation*} 
If  $|x-z-y|\leq 1$,  $|x-z-y|\leq \rho(x-z-y)$ by (P.3). So, we see that 
\begin{equation}\label{4.1}  
 |F(x-z)| \leq C_r\left(\dashint_{B(x-z,\delta)}|F(y)|^r\, dy\right)^{1/r} 
+C_r  \sup_{y:\rho(x-z-y)<\delta}\delta |\nabla F(y)|.  
\end{equation} 
\par 
If $\rho(x-z-y)<\delta$, $\rho(x-y)<\delta+\rho(z)$.  Therefore 
\begin{align*} 
|\nabla F(y)|&\leq \frac{|\nabla F(x+(y-x))|}{(1+\rho(x-y))^N}
(1+\delta+\rho(z))^N 
\\ 
&\leq |\nabla F|_{N,1}^{**}(x)(1+\delta +\rho(z))^N 
\\ 
&\leq 2^N|\nabla F|_{N,1}^{**}(x)(1 +\rho(z))^N.   
\end{align*} 
Thus 
\begin{equation}\label{4.2} 
\sup_{y: \rho(x-z-y)<\delta} 
\delta|\nabla F(y)|\leq 2^N \delta|\nabla F|_{N,1}^{**}(x)(1 +\rho(z))^N.  
\end{equation} 
\par 
On the other hand, 
\begin{align}\label{4.3} 
&\left(\dashint_{B(x-z,\delta)}|F(y)|^r\, dy\right)^{1/r}
\leq \left(\delta^{-\gamma}(\delta+\rho(z))^\gamma 
\dashint_{B(x,\delta+\rho(z))}|F(y)|^r\, dy\right)^{1/r} 
\\ 
&\leq \delta^{-\gamma/r}(\delta+\rho(z))^{\gamma/r} M(|F|^r)(x)^{1/r}     
\notag 
\\ 
&\leq \delta^{-\gamma/r}(1+\rho(z))^{\gamma/r} M(|F|^r)(x)^{1/r}.  \notag 
\end{align} 
\par 
By \eqref{4.1}, \eqref{4.2} and \eqref{4.3}, we have 
\begin{equation*} 
|F(x-z)|\leq C_r \delta^{-\gamma/r}(1+\rho(z))^{\gamma/r} M(|F|^r)(x)^{1/r} + 
2^N C_r \delta|\nabla F|_{N,1}^{**}(x)(1 +\rho(z))^N.  
\end{equation*} 
Thus, if $N=\gamma/r$, we see that 
\begin{equation*} 
\frac{|F(x-z)|}{(1 +\rho(z))^N}\leq C_r \delta^{-N} M(|F|^r)(x)^{1/r} + 
2^N C_r \delta |\nabla F|_{N,1}^{**}(x).  
\end{equation*} 
Taking the supremum in $z$ over $\Bbb R^n$, we get the desired estimate. 
\end{proof}  

\begin{proof}[Proof of Lemma $\ref{Lpar1}$.]    
Let $(T_tf)(x)=f(A_tx)$. Then we note the following. 
\begin{enumerate} 
\item[(T.1)] 
$(T_tF^{**}_{N,R})(x)=(T_tF)^{**}_{N,tR}(x)$. 
\item[(T.2)] 
$T_t(f*g)(x)=t^\gamma(T_tf)*(T_tg)(x)$.  
\item[(T.3)] 
$T_t(M(f))(x)=M(T_tf)(x)$.    
\end{enumerate}
By (T.1) and (T.2) we have 
\begin{equation*}
T_t((f*\varphi_t)^{**}_{N,t^{-1}})(x)=(T_tf*\varphi)^{**}_{N,1}(x). 
\end{equation*} 
Using Lemma \ref{Lpar2}, we see that 
\begin{equation}\label{epar1} 
(T_tf*\varphi)^{**}_{N,1}(x)\leq C\delta^{-N} M(|T_tf*\varphi|^r)(x)^{1/r} +
C\delta |T_tf*\nabla \varphi|^{**}_{N,1}(x). 
\end{equation}
Applying $T_{t^{-1}}$ to \eqref{epar1}, we have 
\begin{equation}\label{epar2} 
(f*\varphi_t)^{**}_{N,t^{-1}}(x)\leq C\delta^{-N} T_{t^{-1}}
(M(|T_tf*\varphi|^r)(x)^{1/r}) +
C\delta T_{t^{-1}}(|T_tf*\nabla \varphi|^{**}_{N,1})(x). 
\end{equation}
From (T.2) and (T.3), it follows that 
\begin{equation}\label{epar3}
T_{t^{-1}}(M(|T_tf*\varphi|^r)(x)^{1/r})=M(|f*\varphi_t|^r)(x)^{1/r}. 
\end{equation}
Also, (T.1) and (T.2) imply 
\begin{equation}\label{epar4} 
T_{t^{-1}}(|T_tf*\nabla \varphi|^{**}_{N,1})(x)= 
|f*(\nabla \varphi)_t|^{**}_{N,t^{-1}}(x). 
\end{equation}
Using \eqref{epar3} and \eqref{epar4} in \eqref{epar2}, we get the conclusion 
of Lemma \ref{Lpar1}.  

\end{proof}
\begin{proof}[Proof of Lemma $\ref{L2.4}$.]   
 Lemma \ref{Lpar1} implies that  
\begin{equation}\label{ineq7} 
F(\varphi,f)(\cdot,t)^{**}_{N, t^{-1}}(x)\leq C\delta^{-N} 
M(|f*\varphi_t|^r)(x)^{1/r} +
C\delta |f*(\nabla\varphi)_t|^{**}_{N,t^{-1}}(x),    
\end{equation}  
where $r=\gamma/N$. 
Applying Lemma \ref{L2.2}, we have 
\begin{multline*}  
|f*(\nabla\varphi)_t|^{**}_{N,t^{-1}}(x)
\\ 
\leq C\sum_{j\geq 0}
C_\varphi(\nabla\varphi, j,N)b^{-jN} F(\varphi,f)(\cdot,b^jt)^{**}
_{N, (b^jt)^{-1}}(x) 
 +CD_\varphi(N) F(\varphi,f)(\cdot,t)^{**}_{N, t^{-1}}(x). 
\end{multline*}   
Thus by \eqref{ineq7} and H\"{o}lder's inequality when $q>1$ we have  
\begin{multline} \label{ineq8} 
F(\varphi,f)(\cdot,t)^{**}_{N, t^{-1}}(x)^q\leq C\delta^{-Nq} 
M(|f*\varphi_t|^r)(x)^{q/r} 
\\ 
+  C\delta^q\sum_{j\geq 0}
C_\varphi(\nabla\varphi, j,N)^qb^{-jNq}b^{-\tau c_qj} 
 F(\varphi,f)(\cdot,b^jt)^{**}_{N, (b^jt)^{-1}}(x)^q        
\\ 
+C\delta^q D_\varphi(N)^q F(\varphi,f)(\cdot,t)^{**}_{N, t^{-1}}(x)^q. 
\end{multline} 
where $\tau>0$,  $c_q=1$ when $q>1$ and $c_q=0$ when $0<q\leq 1$.  
Integrating both sides of the inequality \eqref{ineq8} over $(0,\infty)$ with 
respect to the measure $dt/t$ and applying termwise integration on the right 
hand side, we see that 
\begin{multline}\label{ineq9}  
\int_0^\infty F(\varphi,f)(\cdot,t)^{**}_{N, t^{-1}}(x)^q \, \frac{dt}{t} 
\leq 
C\delta^{-Nq} \int_0^\infty M(|f*\varphi_t|^r)(x)^{q/r}(x)\, \frac{dt}{t} 
\\ 
+ C\delta^q\left[\sum_{j\geq 0}
C_\varphi(\nabla\varphi, j,N)^qb^{-jNq}b^{-\tau c_qj}
+ D_\varphi(N)^q \right]\int_0^\infty 
F(\varphi,f)(\cdot,t)^{**}_{N, t^{-1}}(x)^q \, \frac{dt}{t}.                   
\end{multline}  
By \eqref{ineq5}  the sum in $j$ on the right hand side 
of \eqref{ineq9} is finite.    
By \eqref{beta} and \eqref{beta+} with $\alpha=0$ we easily see that 
the last integral on the right hand side of \eqref{ineq9} is  finite 
with $f\in \mathscr S(\Bbb R^n)$. Thus, along with \eqref{ineq5+}  
 we see that the second term on the right hand side of \eqref{ineq9} 
 is finite. 
So, choosing $\delta$  sufficiently small, we get the conclusion. 
\end{proof} 
We  need the following version of the vector valued inequality for the 
Hardy-Littlewood maximal operator of Fefferman-Stein \cite{FeS} with 
non-isotropic dilations and weights.  
The proof is essentially similar to the one in \cite{FeS}.  
The case $P=E$ is stated in \cite{Sa3} 
(see \cite{RRT} for related results). 

\begin{lemma}\label{L2.6} 
 Let $1<\mu, \nu <\infty$ and $w\in A_\nu$. Then 
 we have 
$$\left\|\left(\int_0^\infty M(F^t)(x)^{\mu} 
\, \frac{dt}{t}\right)^{1/\mu}\right\|_{\nu,w} 
\leq C \left(\int_{\Bbb R^n} \left(\int_0^\infty |F(x,t) |^{\mu}
\, \frac{dt}{t}\right)^{\nu/\mu} w(x)\, dx\right)^{1/\nu}    $$  
for appropriate functions $F(x,t)$ on $\Bbb R^n\times (0, \infty)$
with $F^t(x)=F(x,t)$.   
\end{lemma}  
\begin{proof}[Proof of Theorem $\ref{T2.3}$] 
By a change of variables we may assume that $\hat{\psi}=0$ on 
$\{|\xi|\leq1\}=\{\rho^*(\xi)\leq1\}$.  Define 
$$\hat{\zeta}(\xi)=1-\sum_{j: b^j\leq r_2} \hat{\varphi}(A_{b^j}^*\xi)
\hat{\eta}(A_{b^j}^*\xi).     $$  
Then $\supp(\hat{\zeta})\subset \{\rho^*(\xi)\leq 1\}$, 
$\hat{\zeta}=1$ in $\{\rho^*(\xi)<r_1/r_2\}$. Since $\hat{\psi}=0$ 
on $\{|\xi|\leq 1\}$,  we have 
\begin{align*} 
\hat{\psi}(\xi)&=\sum_{j: b^j\leq r_2}
\hat{\psi}(\xi)\hat{\varphi}(A_{b^j}^*\xi) \hat{\eta}(A_{b^j}^*\xi)
\\ 
&=\sum_{j: b^j\leq r_2} \hat{\varphi}(A_{b^j}^*\xi)
\mathscr F(\alpha^{(b^j)})(A_{b^j}^*\xi),  
\end{align*}
where 
$\alpha^{(b^j)}(x)=(\psi)_{b^{-j}}*\eta(x)$.  
\par 
Thus by an easier version of arguments for the proof of Lemma \ref{L2.2} 
we see that 
\begin{equation*} 
|F(\psi,f)(x,t)|
\leq C\sum_{j: b^j\leq r_2}
C(\psi,j,N)F(\varphi,f)(\cdot,b^jt)^{**}_{N, (b^jt)^{-1}}(x),    
\end{equation*}  
from which it follows that 
\begin{equation*} 
|F(\psi,f)(x,t)|^q 
\leq C\sum_{j: b^j\leq r_2} 
C(\psi,j,N)^q b^{-\tau c_qj} 
F(\varphi,f)(\cdot,b^jt)^{**}_{N, (b^jt)^{-1}}(x)^q,    
\end{equation*}  
where $\tau>0$ and $c_q$ is as in \eqref{ineq8}. 
Integrating with the measure $dt/t$ over $(0,\infty)$, we have 
\begin{multline}\label{ineq6}  
\int_0^\infty |F(\psi,f)(x,t)|^q\, \frac{dt}{t} 
\\ 
\leq C\left[\sum_{j: b^j\leq r_2}
C(\psi,j,N)^q b^{-\tau c_qj}\right]  \int_0^\infty 
F(\varphi,f)(\cdot,t)^{**}_{N, t^{-1}}(x)^q \, \frac{dt}{t}.  
\end{multline}  
By \eqref{ineq4} the series on the right hand side of 
\eqref{ineq6}  converges. 
Let $0<p, q<\infty$ and $w \in A_{\infty}$. If $N$ is sufficiently large 
so that $r=\gamma/N<q, p$ and $w\in A_{pN/\gamma}$, 
from \eqref{ineq6} and Lemma \ref{L2.4} it follows  that 
\begin{align}\label{ineq11} 
\left\|\left(\int_0^\infty|f*\psi_t|^q\, \frac{dt}{t}\right)^{1/q}
\right\|_{p,w}
&\leq C\left\|\left(\int_0^\infty M(|f*\varphi_t|^r)(x)^{q/r} 
\, \frac{dt}{t}\right)^{1/q}\right\|_{p,w}
\\  
&= C\left\|\left(\int_0^\infty M(|f*\varphi_t|^r)(x)^{q/r} 
\, \frac{dt}{t}\right)^{r/q}\right\|_{p/r,w}^{1/r}            \notag 
\\ 
&\leq C\left\|\left(\int_0^\infty |f*\varphi_t(x)|^q  
\, \frac{dt}{t}\right)^{1/q}\right\|_{p,w},            \notag 
\end{align}  
where the last inequality follows form Lemma \ref{L2.6}.   
This completes the proof of Theorem \ref{T2.3}.  
\end{proof} 

To conclude this section,  we give a proof of the following results 
used above. 
\begin{lemma}\label{L2.8} 
Let $\varphi\in B$.  Suppose that $\psi\in L^1$, $\hat{\psi}\in 
C^\infty(\Bbb R^n\setminus\{0\})$ and  
we have the estimates \eqref{beta+} with $\hat{\psi}$ in place of 
$\hat{\varphi}$ for all multi-indices $\alpha$ and  all $\tau>0$.  
Let $L, J\geq 0$. 
Then 
\begin{enumerate} 
\item[$(1)$]  
$\sup_{j: b^j\leq J}C_\varphi(\psi,j,L)b^{-j\tau}<\infty$ for any 
$\tau>0$, where $C_\varphi(\psi,j,L)=C(\psi,j,L)$ is as in \eqref{c} with 
$\psi$ in place of $\partial_k\varphi$.    
\item[$(2)$] 
$ D_\varphi(\Xi_k, L)<\infty$,  $1\leq k\leq n$, 
where $\Xi_k(\xi)=2\pi i\xi_k$ as above and    
$D_\varphi(\Xi_k, L)=D(\Xi_k, L)$ is as in \eqref{d2}. 
\end{enumerate}
\end{lemma}  
\begin{proof} 
To prove part (1), since $1+\rho(x)\leq c(1+|x|)$ by (P.4),  we have  
\begin{multline*} 
(1+|x|)^{[n/2]+1}C_0(\psi,t,L,x) 
\\ 
\leq C\left|\int \hat{\psi}(A_{t^{-1}}^*\xi)\hat{\eta}(\xi)e^{2\pi i
\langle x, \xi\rangle}\, d\xi\right|+
C\sup_{|\alpha|=L+[n/2]+1} 
\left|\int \partial_\xi^\alpha\left[\hat{\psi}(A_{t^{-1}}^*\xi)
\hat{\eta}(\xi)\right]e^{2\pi i\langle x, \xi\rangle}\, d\xi\right|,     
\end{multline*}  
where $C_0(\psi,t,L,x)$ is as in \eqref{czero} with 
$\psi$ in place of $\partial_k\varphi$ and $[a]$ denotes the largest 
integer not exceeding $a$.     We recall that 
$\hat{\eta} \in C^{\infty}(\Bbb R^n)$ with support in 
$\{r_1<\rho^*(\xi)<r_2\}$, which is in Lemma \ref{L2.1}.  
It is known that $\|A_{t^{-1}}^*\|\leq t^{-\kappa}$ for $t\in (0,1]$ with 
some $\kappa\geq 1$. Thus   
 \eqref{beta+} for $\hat{\psi}$ implies 
$$\left|\partial_\xi^\alpha \left[\hat{\psi}(A_{t^{-1}}^*\xi)
\hat{\eta}(\xi)\right]\right|\leq C_{\alpha,M}t^{\tau}, \quad 0<t\leq M, $$  
for any $M>0$, 
if $|\alpha|= L+[n/2]+1$ or $\alpha=0$.  Thus 
\begin{equation*} 
C_0(\psi,t,L,x)\leq C(1+|x|)^{-[n/2]-1}G(x) 
\end{equation*}
with some $G\in L^2$ such that $\|G\|_2\leq Ct^{\tau}$, and hence,  
the Schwarz inequality implies  
\begin{equation}\label{ineq12}
\int_{\Bbb R^n}C_0(\psi,t,L,x)\, dx \leq  Ct^{\tau},   
\end{equation} 
since $[n/2]+1>n/2$.  From \eqref{ineq12} with $t=b^j$ we obtain the 
conclusion of part (1). 
\par 
Similarly, we can see that 
\begin{equation*}\label{ineq13}
\int_{\Bbb R^n}D(\Xi_k,L,x)\, dx <\infty,     
\end{equation*} 
where  $D(\Xi_k,L,x)$ is as in \eqref{d}.  
This completes the proof of part (2).   
\end{proof} 
 
\section{Littlewood-Paley functions and parabolic Hardy spaces}  

As an application of Theorem \ref{T2.3} we have the following. 
\begin{corollary}\label{C3.1} 
Let $0<p\leq 1$. 
Suppose that $\varphi \in B$. 
Then we have 
$$ \|f\|_{H^p}\leq C_p\|g_\varphi(f)\|_p  $$  
for $f\in H^p(\Bbb R^n)\cap \mathscr S(\Bbb R^n)$ with a positive constant 
$C_p$ independent of $f$. 
\end{corollary} 
\par 
Let $\mathscr H$ be the Hilbert space of 
functions $u(t)$ on $\Bbb R_+$ such that $\|u\|_{\mathscr H}=
\left(\int_0^\infty|u(t)|^2\, dt/t\right)^{1/2}<\infty$.  
Let $L^q_{\mathscr H}(\Bbb R^n)$ be the Lebesgue space  of functions 
$h(y,t)$ with the norm  
$$\|h\|_{q,\mathscr H}=\left(\int_{\Bbb R^n}\|h^y\|_{\mathscr H}^q \, dy
\right)^{1/q}, $$ 
where $h^y(t)=h(y,t)$. 
\par 
Let $0<p\leq 1$. 
We consider the parabolic Hardy space of functions on $\Bbb R^n$ with values 
in $\mathscr H$, which is denoted by 
$H^p_{\mathscr H}(\Bbb R^n)$.  
Choose $\Phi\in \mathscr S(\Bbb R^n)$ as in the definition of $H^p$ in Section 
1.  Let $h \in L^2_{\mathscr H}(\Bbb R^n)$.  
We say  $h \in H^p_{\mathscr H}(\Bbb R^n)$ if 
$\|h\|_{H^p_{\mathscr H}}= \|h^*\|_{L^p}<\infty$  
with 
$$h^*(x)=\sup_{s>0}\left(\int_0^\infty |\Phi_s*h^{t}(x)|^2\, \frac{dt}{t}
\right)^{1/2}, $$where we write $h^{t}(x)=h(x,t)$. 
\par 
To prove Corollary \ref{C3.1} we need the following. 
\begin{lemma}\label{L3.2} 
Let $\hat{\psi}$ be a function of $\mathscr S(\Bbb R^n)$ with support in 
$\{1\leq \rho^*(\xi)\leq 2\}$.   Suppose that 
$$\int_0^\infty|\hat{\psi}(A_t^*\xi)|^2\, \frac{dt}{t}=1 \quad \text{for all 
$\xi\neq 0$.} 
$$  
Let $F(y,t)=f*\psi_t(y)$ with $f\in H^p(\Bbb R^n)\cap \mathscr S(\Bbb R^n)$, 
 $0<p\leq 1$.  
Then $F\in H^p_{\mathscr H}(\Bbb R^n)$  and
\begin{equation*}\label{}  
\|f\|_{H^p}\leq   C\|F\|_{H^p_{\mathscr H}}.     
\end{equation*} 
\end{lemma} 
Define 
\begin{equation*}
E_\psi^\epsilon(h)(x)=\int_0^\infty\int_{\Bbb R^n} 
\psi_t(x-y)h_{(\epsilon)}(y,t)\,dy\,\frac{dt}{t}, 
\end{equation*} 
where  $h\in L^2_{\mathscr H}$ and $h_{(\epsilon)}(y,t)=h(y,t)\chi_{(\epsilon, \epsilon^{-1})}(t)$, $0<\epsilon<1$,  and 
we assume that $\psi\in L^1(\Bbb R^n)$ with the cancellation \eqref{cancell}.  
Here $\chi_S$ denotes the characteristic function of a set $S$.  
\par 
We apply the following result in proving Lemma \ref{L3.2}.  
\begin{lemma}\label{L3.3} 
 Suppose that $\psi\in \mathscr S(\Bbb R^n)$ and $\supp \hat{\psi}\subset 
 \{1\leq \rho^*(\xi)\leq 2\}$. 
Then 
$$\sup_{\epsilon\in (0,1)}
\|E_\psi^\epsilon(h)\|_{H^p}\leq C\|h\|_{H^p_{\mathscr H}}, \quad 
0 <p\leq 1.   $$  
\end{lemma} 
Let $a$ be a $(p, \infty)$ atom in  $H^p_{\mathscr H}(\Bbb R^n)$:    
\begin{enumerate} 
\item[(i)] $\left(\int_0^\infty |a(x,t)|^2\, dt/t \right)^{1/2} \leq 
|B|^{-1/p}$, where $B$ is a ball in $\Bbb R^n$ with respect to $\rho$; 
\item[(ii)] $\supp(a(\cdot, t))\subset B$ for all $t>0$, where $B$ is 
 as in $(i);$  
\item[(iii)] $\int_{\Bbb R^n} a(x,t)x^\alpha \, dx=0$ for all $t>0$ and 
$\alpha$ such that $|\alpha|\leq [\gamma(1/p -1)]$, where    
$\alpha=(\alpha_1, \dots, \alpha_n)$ is a multi-index with 
$x^\alpha=x_1^{\alpha_1}\dots x_n^{\alpha_n}$. 
\end{enumerate}  
\par 
To prove Lemma \ref{L3.3} we use the following atomic decomposition.  
\begin{lemma}\label{L3.4}
Let $h\in L^2_{\mathscr H}(\Bbb R^n)$. Suppose that
 $h\in H^p_{\mathscr H}(\Bbb R^n)$. 
Then we can find a sequence $\{a_k\}$ of 
$(p,\infty)$ atoms in $H^p_{\mathscr H}(\Bbb R^n)$ and a sequence 
$\{\lambda_k\}$ of positive numbers such that 
$h=\sum_{k=1}^\infty\lambda_k a_k$ in $H^p_{\mathscr H}(\Bbb R^n)$ and in 
$L^2_{\mathscr H}(\Bbb R^n)$ and such that 
$\sum_{k=1}^\infty\lambda_k^p\leq C\|h\|_{H^p_{\mathscr H}}^p$ with 
a constant $C$ independent of $h$.   
\end{lemma}   
 We can find in \cite{LU} a proof of the atomic 
decomposition for $H^p(\Bbb R^n)$ (see also \cite{C2}, \cite{GR} and 
\cite{ST}).  
The vector valued case can be treated similarly.     
\begin{proof}[Proof of Lemma $\ref{L3.3}$]     
Let $a$ be a $(p, \infty)$ atom in  $H^p_{\mathscr H}(\Bbb R^n)$ supported 
on the ball $B$ of the definition of the atom.  
We choose a non-negative $C^\infty$ function $\Phi$ on $\Bbb R^n$ supported 
on $\{|x|<1\}$ with $\int \Phi(x)\, dx=1$. 
We prove 
\begin{equation}\label{a3}
\int_{\Bbb R^n} \sup_{s>0}\left|\Phi_s*E_\psi^\epsilon(a)(x)\right|^p
\, dx \leq C. 
\end{equation} 
To show this, 
by applying translation and dilation arguments, we may assume that $B=B(0,1)$. 
Let $\widetilde{B}=B(0,2)$. Then $2\rho(y)\leq \rho(x)$  
if $y\in B$ and $x\in \Bbb R^n\setminus \widetilde{B}$.  
Let $\Psi_{s,t}=\Phi_s*\psi_t$, $s, t>0$. We note that 
$\Phi_s*\psi_t=(\Phi_{s/t}*\psi)_t$ and $\Phi_{u}*\psi$, $u>0$,   belongs to 
a bounded subset of the topological vector space $\mathscr S(\Bbb R^n)$, 
which can easily be seen by noting that $\mathscr F(\Phi_{u}*\psi)(\xi) 
=\hat{\Phi}(A_u^*\xi)\hat{\psi}(\xi)$ and recalling that $\hat{\psi}(\xi)$ 
is supported on $\{1\leq \rho^*(\xi)\leq 2\}$.   
\par 
Let $P_x(y)$ be the Taylor polynomial in $y$ of order $M=[\gamma(1/p -1)]$ 
at $0$ for $\Phi_{s/t}*\psi(x-y)$. 
Then, if $\rho(x)\geq 2\rho(y)$ and we have 
$$|\Phi_{s/t}*\psi(x-y)-P_x(y)|
\leq C|y|^{M+1}(1+\rho(x))^{-L},   $$  
where $L$ is sufficiently large, which will be specified below, 
 and the constant $C$ is independent of $s, t, x, y$. 
This  implies  
$$|\Psi_{s,t}(x-y)-t^{-\gamma}P_{A_{t^{-1}}x}(A_{t^{-1}}y)|\leq 
Ct^{-\gamma}|A_{t^{-1}}y|^{M+1}(1+\rho(x)/t)^{-L}.   $$  
Thus, using   the properties of an atom and the Schwarz 
inequality, for $x\in \Bbb R^n\setminus \widetilde{B}$ we see that 
\begin{align}\label{e3.0} 
&\left|\Phi_s*E_\psi^\epsilon(a)(x)\right|=\left|\iint 
\left(\Psi_{s,t}(x-y)-t^{-\gamma}P_{A_{t^{-1}}x}(A_{t^{-1}}y) \right)
a_{(\epsilon)}(y,t)\, dy\, \frac{dt}{t}\right|  
\\ 
&\leq \int_B\left(\int_0^\infty\left|\Psi_{s,t}(x-y)-t^{-\gamma}
P_{A_{t^{-1}}x}(A_{t^{-1}}y) 
\right|^2 \, \frac{dt}{t}\right)^{1/2}\left(\int_0^\infty|a(y,t)|^2
\, \frac{dt}{t}\right)^{1/2}\, dy       \notag 
\\ 
&\leq C|B|^{-1/p}\int_{B}\left(\int_0^\infty
\left(t^{-\gamma}|A_{t^{-1}}y|^{M+1}(1+\rho(x)/t)^{-L}\right)^2 \, 
\frac{dt}{t}\right)^{1/2}\, dy.    \notag 
\end{align}  
\par 
Now we show that 
\begin{equation}\label{e3.01}
I(x,y):=\int_0^\infty
\left(t^{-\gamma}|A_{t^{-1}}y|^{M+1}(1+\rho(x)/t)^{-L}\right)^2 \, 
\frac{dt}{t} \leq C\rho(x)^{-2(\gamma +M+1)} 
\end{equation} 
for $y\in B$, $x\in \Bbb R^n\setminus \widetilde{B}$, 
 if $L$ is sufficiently large.  We first see that  
\begin{equation}\label{e3.02} 
I(x,y)=\rho(x)^{-2\gamma}\int_{0}^\infty
\left(t^{-\gamma}|A_{(\rho(x)t)^{-1}}y|^{M+1}(1+t^{-1})^{-L}\right)^2 \, 
\frac{dt}{t}. 
\end{equation}  
By (P.6) we have 
\begin{align}\label{e3.03} 
&\int_{\rho(x)^{-1}}^\infty 
\left(t^{-\gamma}|A_{(\rho(x)t)^{-1}}y|^{M+1}(1+t^{-1})^{-L}\right)^2 \, 
\frac{dt}{t}
\\ \notag 
&\leq C|y|^{2(M+1)}\rho(x)^{-2(M+1)}\int_{\rho(x)^{-1}}^\infty 
t^{-2(\gamma+M+1)} (1+t^{-1})^{-2L}\, \frac{dt}{t} 
\\ \notag 
&\leq C|y|^{2(M+1)}\rho(x)^{-2(M+1)}   
\end{align} 
if $L>\gamma +M+1$. If $s\geq 1$, $|A_sy|\leq Cs^\kappa|y|$ for some 
$\kappa\geq 1$.  Thus 
\begin{align}\label{e3.04} 
&\int_0^{\rho(x)^{-1}}
\left(t^{-\gamma}|A_{(\rho(x)t)^{-1}}y|^{M+1}(1+t^{-1})^{-L}\right)^2 \, 
\frac{dt}{t}
\\ \notag 
&\leq C|y|^{2(M+1)}\rho(x)^{-2\kappa(M+1)}\int_0^{\rho(x)^{-1}}  
t^{-2(\gamma+\kappa(M+1))} (1+t^{-1})^{-2L}\, \frac{dt}{t} 
\\ \notag 
&\leq C|y|^{2(M+1)}\rho(x)^{-2\kappa(M+1)},   
\end{align} 
if $L>\gamma +\kappa(M+1)$. By \eqref{e3.02}, \eqref{e3.03} and 
\eqref{e3.04} we obtain \eqref{e3.01}.  
From \eqref{e3.0} and \eqref{e3.01} we have 
\begin{equation*} 
\left|\Phi_s*E_\psi^\epsilon(a)(x)\right| \leq C\rho(x)^{-(\gamma +M+1)}
\end{equation*} 
for $x\in \Bbb R^n\setminus \widetilde{B}$.  
\par 
Since $p>\gamma/(\gamma+M+1)$, it follows that   
\begin{equation}\label{a1}
\int_{\Bbb R^n\setminus \widetilde{B}} \sup_{s>0}
\left|\Phi_s*E_\psi^\epsilon(a)(x)
\right|^p\, dx 
\leq C\int_{\Bbb R^n\setminus \widetilde{B}}
\rho(x)^{-p(\gamma+M+1)} \, dx\leq C  
\end{equation} 
(see \cite{Ca}).  
\par 
Using $\int_0^\infty|\hat{\psi}(A_t^*\xi)|^2\, dt/t\leq C$, by duality 
we can easily  see that 
$$\sup_{\epsilon\in(0,1)}\|E_\psi^\epsilon(h)\|_2 
\leq C\|h\|_{L^2_\mathscr H}, \quad h\in L^2_{\mathscr H}(\Bbb R^n).$$ 
So, by H\"{o}lder's inequality and the properties (i), (ii) of $a$, we get 
\begin{align}\label{a2}
\int_{\widetilde{B}} \sup_{s>0}\left|\Phi_s*E_\psi^\epsilon(a)(x)
\right|^p\, dx 
&\leq C\left(\int_{\widetilde{B}}|M(E_\psi^\epsilon(a))(x)|^2
\, dx \right)^{p/2} 
\\ 
&\leq C\left(\int_{B}\int_0^\infty |a(y,t)|^2\, \frac{dt}{t} \, dy 
\right)^{p/2}    \notag 
\\ 
&\leq C.    \notag
\end{align} 
Combining \eqref{a1} and \eqref{a2}, we have \eqref{a3}. 
By Lemma \ref{L3.4} and \eqref{a3} we can prove 
 \begin{equation*}\label{h1}
\int_{\Bbb R^n} \sup_{s>0}\left|\Phi_s*E_\psi^\epsilon(h)(x)\right|^p
\, dx \leq C\|h\|_{H^p_{\mathscr H}}^p.   
\end{equation*} 
This completes the proof. 
\end{proof}   
  
\begin{proof}[Proof of Lemma $\ref{L3.2}$.]  
By using the atomic decomposition for $H^p(\Bbb R^n)$, 
we can prove the fact that $F\in H^p_{\mathscr H}(\Bbb R^n)$ similarly 
to the proof of Lemma \ref{L3.3}  (see \cite[Lemma 3.6]{U}). 
\par  
We note that 
\begin{equation*} 
E^\epsilon_{\widetilde{\bar{\psi}}}(F)(x)=\int_\epsilon^{\epsilon^{-1}}
\int_{\Bbb R^n} \psi_t*f(y)\bar{\psi}_t(y-x)\,dy\,\frac{dt}{t} 
= \int_{\Bbb R^n} \Psi^{(\epsilon)}(x-z)f(z)\, dz, 
\end{equation*} 
where 
\begin{equation*} 
\Psi^{(\epsilon)}(x)=
\int_\epsilon^{\epsilon^{-1}}\int_{\Bbb R^n}  
\psi_t(x+y)\bar{\psi}_t(y)\,dy\,\frac{dt}{t}. 
\end{equation*} 
We have  
\begin{equation*} 
\widehat{\Psi^{(\epsilon)}}(\xi)
=\int_\epsilon^{\epsilon^{-1}}\hat{\psi}(A_t^*\xi)
\widehat{\bar{\psi}}(-A_t^*\xi)
\,\frac{dt}{t}=\int_\epsilon^{\epsilon^{-1}}|\hat{\psi}(A_t^*\xi)|^2
\,\frac{dt}{t}. 
\end{equation*} 
This and Lemma \ref{L3.3} imply  
\begin{equation*}\label{}  
\|f\|_{H^p}\leq  
 C\liminf_{\epsilon\to 0}\|E^\epsilon_{\widetilde{\bar{\psi}}}(F)\|_{H^p}
\leq C\|F\|_{H^p_{\mathscr H}}.   
\end{equation*} 
\end{proof} 
We also need the following result to prove Corollary \ref{C3.1}.  
\begin{lemma}\label{L3.5} 
Let $\eta\in \mathscr S(\Bbb R^n)$ satisfy $\supp(\hat{\eta})\subset 
\{1/2\leq \rho^*(\xi)\leq 4\}$ and $\hat{\eta}(\xi)=1$ on 
$\{1\leq \rho^*(\xi)\leq 2\}$.  Let $\psi$ be as in Lemma $\ref{L3.2}$.   
Suppose that $\Phi\in \mathscr S(\Bbb R^n)$ satisfies 
$\int_{\Bbb R^n} \Phi(x)\, dx=1$ and $\supp(\Phi)\subset B(0,1)$. 
Then for $p, q>0$ and $f\in  \mathscr S(\Bbb R^n)$ we have 
$$\left\|\left(\int_0^\infty\sup_{s>0}|\Phi_s*\psi_t*f|^q\, \frac{dt}{t}
\right)^{1/q}\right\|_{p} 
\leq C\left\|\left(\int_0^\infty|\eta_t*f|^q\, 
\frac{dt}{t}\right)^{1/q}\right\|_{p}. $$  
\end{lemma} 
\begin{proof} 
Since $\hat{\Phi}(A_s^*\xi)\hat{\psi}(A_t^*\xi)=\hat{\Phi}(A_s^*\xi)
\hat{\psi}(A_t^*\xi)\hat{\eta}(A_t^*\xi)$, we have 
\begin{align*} 
|\Phi_s*\psi_t*f(x)|&\leq (f*\eta_t)^{**}_{N, t^{-1}}(x)
\int_{\Bbb R^n}|\Phi_s*\psi_t(w)|(1+t^{-1}\rho(w))^N\, dw 
\\ 
&= (f*\eta_t)^{**}_{N, t^{-1}}(x) 
\int_{\Bbb R^n}|\Phi_{s/t}*\psi(w)|(1+\rho(w))^N\, dw 
\\ 
&\leq C_N (f*\eta_t)^{**}_{N, t^{-1}}(x)  
\end{align*} 
for any $N>0$, 
where $C_N$ is independent of $s, t$. This follows from the observation that 
$\Phi_{s/t}*\psi$, $s, t>0$, belongs to a bounded subset of the 
topological vector space $\mathscr S(\Bbb R^n)$, 
as in the proof of Lemma \ref{L3.3}. 
Thus 
\begin{equation}\label{ineq3.1}  
\left(\int_0^\infty\sup_{s>0}|\Phi_s*\psi_t*f(x)|^q\, \frac{dt}{t}
\right)^{1/q}
\leq C\left(\int_0^\infty|(f*\eta_t)^{**}_{N, t^{-1}}(x)|^q\, 
\frac{dt}{t}\right)^{1/q}.  
\end{equation}
By \eqref{ineq3.1} and Lemma \ref{L2.4} with $\eta$ in place of $\varphi$, 
we have 
\begin{equation*} 
\left(\int_0^\infty\sup_{s>0}|\Phi_s*\psi_t*f(x)|^q\, \frac{dt}{t}
\right)^{1/q}
\leq C\left(\int_0^\infty M(|f*\eta_t|^r)(x)(x)^{q/r}
\, \frac{dt}{t}\right)^{1/q}, 
\end{equation*} 
where $N=\gamma/r$.  This and Lemma \ref{L2.6} prove Lemma \ref{L3.5} as 
in \eqref{ineq11}.   
\end{proof} 
\begin{proof}[Proof of Corollary $\ref{C3.1}$.]  
Let $\eta$ be as in Lemma \ref{L3.5}. 
Applying Lemma \ref{L3.2} and Lemma \ref{L3.5} with $q=2$ and $p\in (0,1]$,
 we see that 
$$\|f\|_{H^p}\leq  C\left\|g_\eta(f)\right\|_p, \quad 
f\in H^p(\Bbb R^n)\cap \mathscr S(\Bbb R^n), $$  
which combined with Theorem \ref{T2.3} with $q=2$, $p\in (0,1]$, 
$w=1$ and with $\eta$ in place of $\psi$ proves Corollary \ref{C3.1}. 

\end{proof} 

\begin{proof}[Proof of Theorem $\ref{T1.1}$] 
Let $\varphi$ be as in Theorem \ref{T1.1}. Let $0<p\leq 1$. 
The inequality 
\begin{equation*} 
\|g_\varphi(f)\|_p\leq C\|f\|_{H^p} 
\end{equation*} 
can be proved similarly to the proof of the statement 
$F\in H^p_{\mathscr H}(\Bbb R^n)$ in Lemma \ref{L3.2} for 
$f\in H^p(\Bbb R^n)\cap \mathscr S(\Bbb R^n)$ by using 
the atomic decomposition 
for $H^p(\Bbb R^n)$.  This and Corollary \ref{C3.1} imply  
\begin{equation*}
c_1\|f\|_{H^p}\leq \|g_{\varphi}(f)\|_p\leq c_2\|f\|_{H^p}  
\end{equation*}
for $f\in H^p(\Bbb R^n)\cap \mathscr S(\Bbb R^n)$,  from which 
the conclusion of Theorem \ref{T1.1} follows by arguments similar to the one 
in \cite[pp. 149--150]{U}, since 
$H^p(\Bbb R^n)\cap \mathscr S(\Bbb R^n)$ is dense in $ H^p(\Bbb R^n)$ 
(see \cite{CT2}).  
\end{proof}
\par 
It is not difficult to see that we have discrete parameter versions of 
Theorems \ref{T1.1} and \ref{T2.3}. 
To conclude this note we remark the following results.  

\begin{theorem}\label{T3.6} 
Let $\varphi$ be as in Theorem $\ref{T1.1}$ and $0<p\leq 1$. 
Then, there exist positive constants $c_1, c_2$ such that   
$$c_1\|f\|_{H^p}\leq \left\|\left(\sum_{j=-\infty}^\infty|f*\varphi_{b^j}|^2
\right)^{1/2}\right\|_p\leq c_2\|f\|_{H^p}  $$  
\end{theorem}  
for $f\in H^p(\Bbb R^n)$. 

\begin{theorem}\label{T3.7} 
Let  $0<p, q<\infty$ and $w \in A_{\infty}$. 
Suppose that $\varphi$ and $\psi$ fulfill the hypotheses of Theorem  
$\ref{T2.3}$.  Then we have, for $f\in \mathscr S(\Bbb R^n)$, 
$$\left\|\left(\sum_{j=-\infty}^\infty|f*\psi_{b^j}|^q\right)^{1/q}
\right\|_{p,w} 
\leq 
C\left\|\left(\sum_{j=-\infty}^\infty|f*\varphi_{b^j}|^q
\right)^{1/q}\right\|_{p,w}.   
$$
\end{theorem}

\end{document}